\documentclass[11pt]{scrartcl}

\usepackage{kpfonts}
\usepackage[english]{babel}
\usepackage{amsmath,amssymb,amsthm}
\usepackage{tkz-graph}
\usepackage{enumerate}
\newtheorem{theorem}{Theorem}[section]
\newtheorem{proposition}[theorem]{Proposition}
\newtheorem{corollary}[theorem]{Corollary}
\theoremstyle{remark}
\newtheorem{remark}[theorem]{Remark}

\title{Lower Bounds for the Probability of a Union via Chordal Graphs}
\author{Klaus Dohmen\footnote{Address: Department of Mathematics, Mittweida
    University of Applied Sciences, Postfach 1451, 09644 Mittweida,
    Germany. E-mail: \texttt{dohmen@hsmw.de}. WWW:
    \texttt{http://www.hsmw.de/dohmen}.}}

\begin{document}

\maketitle

\begin{abstract}
We establish new Bonferroni-type lower bounds for the probability of a union
of finitely many events where the selection of intersections in the estimates
is determined by the clique complex of a chordal graph.
\end{abstract}

\section{Introduction}

The classical Bonferroni inequalities state that for any finite collection of
events $\{A_v\}_{v\in V}$,
\begin{equation}
\label{classicalbon}
\begin{array}{rcl}
\displaystyle \Pr\left( \bigcup_{v\in V} A_v \right) & \le & \displaystyle\sum_{\substack{I\in\mathscr{P}^\ast(V) \\ |I|\le 2r-1}} (-1)^{|I|-1} \Pr\left( \bigcap_{i\in I} A_i \right)\\
\displaystyle \Pr\left( \bigcup_{v\in V} A_v \right) & \ge & \displaystyle\sum_{\substack{I\in\mathscr{P}^\ast(V) \\ |I|\le 2r}} (-1)^{|I|-1} \Pr\left( \bigcap_{i\in I} A_i \right)
\end{array}
\qquad (r=1,2,3,\dots),
\end{equation}
where $\mathscr{P}^\ast(V)$ denotes the set of non-empty subsets of $V$.
Numerous variants of these inequalities are known; see, e.g., Galambos and
Simonelli \cite{Galambos-Simonelli:1996} for a survey on
\emph{Bonferroni-type} inequalities, which are variants of
(\ref{classicalbon}) that are applicable to any finite family of events.  More
recent variants arising from abstract tubes
\cite{Dohmen:2003,Naiman-Wynn:1997} and monomial ideals \cite{Saenz-Wynn:2009}
take advantage of the underlying structure of events, thus reducing the sum in
(\ref{classicalbon}) to a subcomplex of $\mathscr{P}^\ast(V)$, while providing
provably tighter bounds at all truncations.

This short note is inspired by both lines of research.  Our main result is a
variant of the lower bound in (\ref{classicalbon}), which is applicable to any
finite family of events, and where the selection of intersections in the
estimates is determined by the clique complex of a chordal graph.  

We refer to \cite{Diestel:2005} for terminology on graphs.  We write $G=(V,E)$
to denote that $G$ is a graph having vertex-set $V$, which we assume to be
finite, and edge-set $E$, consisting of two-element subsets of $V$.  We use
$\mathscr{C}(G)$ to denote the \emph{clique complex} of $G$, that is, the
abstract simplicial complex consisting of all non-empty cliques of $G$.  A
graph is referred to as \emph{chordal} if it contains no cycles of length four
or any higher length as induced subgraphs.

Our main result in Section \ref{sec:main-result} complements the upper bound
in Proposition \ref{cgs} below, which interpolates between Boole's inequality
($G$ edgeless or $r=1$), Hunter's inequality \cite{Hunter:1976} ($G$ a tree
and $r\ge 2$), and the sieve formula ($G$ complete and $r\ge (|V|+1)/2$).  For
$r=2$, the bound in Proposition \ref{cgs} was independently found by Boros and
Veneziani \cite{Boros-Veneziani:2002} by linear programming techniques; see
also \cite{Veneziani:2009} for a discussion on this particular case.

\begin{proposition}[\cite{Dohmen:2003,Dohmen:2002}]
\label{cgs}
Let $\{A_v\}_{v\in V}$ be a finite collection of events, where the indices
form the vertices of a chordal graph $G$.  Then,
\begin{align}
\label{alt1}
\Pr\left(\, \bigcup_{v\in V} A_v \right) & \,\le\,
\sum_{\substack{I\in\,\mathscr{C}(G)\\ |I|\le 2r-1}} (-1)^{|I|-1} \,\Pr\left(
\,\bigcap_{i\in I} A_i \right) \quad (r=1,2,3,\dots) .
\end{align}
\end{proposition}

\section{Main result}
\label{sec:main-result}

We use $c(G)$ resp\@. $\alpha(G)$ to denote the number of connected components
resp\@. the independence number of $G$.  Our main result is a lower
bound analogue of (\ref{alt1}).

\begin{theorem}
\label{maintheorem}
Let $\{A_v\}_{v\in V}$ be a non-empty finite collection of events, where the
indices form the vertices of a chordal graph $G$.  Then,
\begin{align}
\label{mainineq}
\Pr\left(\, \bigcup_{v\in V} A_v \right) & \,\ge\, \frac{1}{\alpha(G)}
\sum_{\substack{I\in\,\mathscr{C}(G)\\ |I|\le 2r}} (-1)^{|I|-1} \,\Pr\left(
\,\bigcap_{i\in I} A_i \right) \quad (r=1,2,3,\dots) .
\end{align}
\end{theorem}

\begin{remark}
Due to (\ref{alt1}) the upper bound in (\ref{mainineq}) is non-negative for
$r\ge |V|/2$. If $G$ is complete, (\ref{mainineq}) coincides with the
classical Bonferroni lower bounds.
\end{remark}

The proof of Theorem \ref{maintheorem} is facilitated by the following
proposition.

\begin{proposition}
\label{myprop}
For any chordal graph $G=(V,E)$, 
\begin{equation}
\sum_{I\in\mathscr{C}(G)\atop |I|\le 2r} (-1)^{|I|-1} \,\le\, c(G) \quad (r=1,2,3,\dots),
\end{equation}
with equality if $r\ge |V|/2$.
\end{proposition}

\begin{proof}
If $G$ is connected, the result follows from the following topological
results: 
\begin{enumerate}[i)]
\item The clique complex of any connected chordal graph is contractible
(\cite{Edelman-Reiner:2000}).
\item For any contractible abstract simplicial complex, the
Euler characteristic of its $(2r-1)$-skeleton is at most $1$, with equality if
$r\ge |V|/2$ (\cite{Naiman-Wynn:1997}).
\end{enumerate}
If $G$ is disconnected, apply the inequality to each of its connected
components.
\end{proof}

\begin{proof}[Proof of Theorem \ref{maintheorem}]
For any $J\subseteq V$, $J\neq \emptyset$, define
\begin{equation*}
B_J := \bigcap_{i\in J} A_i \cap \bigcap_{i\notin J} \overline{A_i} .
\end{equation*}
Note that the $B_J$'s form a partition of $\bigcup_{v\in V} A_v$.  Clearly,
$G[J]$ is chordal and $c(G[J])\le \alpha(G)$ for any $J\subseteq V$,
$J\neq\emptyset$.  Hence, by applying Proposition \ref{myprop} to $G[J]$ we
obtain
\begin{align*}
\Pr \left( \bigcup_{v\in V} A_v \right) & = \Pr \left(
\bigcup_{\substack{J\subseteq V\\ J\neq\emptyset}} B_J \right) \,=\,
\sum_{\substack{J\subseteq V\\ J\neq\emptyset}} \Pr(B_J)
\,\ge\, \frac{1}{\alpha(G)} \sum_{\substack{J\subseteq V \\ J\neq \emptyset}} c(G[J]) \Pr(B_J) \\
& \ge \frac{1}{\alpha(G)}\sum_{\substack{J\subseteq V \\ J\neq \emptyset}}
\sum_{\substack{I\in \mathscr{C}(G[J])\\ |I|\le 2r}} (-1)^{|I|-1} \Pr(B_J)
\, = \,\frac{1}{\alpha(G)} \sum_{\substack{I\in\mathscr{C}(G)\\ |I|\le 2r}} (-1)^{|I|-1} \sum_{J\supseteq I} \Pr(B_J) \\
& = \frac{1}{\alpha(G)}\sum_{\substack{I\in\mathscr{C}(G)\\ |I|\le 2r}}
(-1)^{|I|-1} \Pr \left( \bigcup_{J\supseteq I} B_J \right) \, = \,
\frac{1}{\alpha(G)}\sum_{\substack{I\in\mathscr{C}(G)\\ |I|\le 2r}}
(-1)^{|I|-1} \Pr\left( \bigcap_{i\in I} A_i \right). \qedhere
\end{align*}
\end{proof}

\begin{remark}
In view of the preceding proof and Proposition \ref{myprop} the best bound in
(\ref{mainineq}) is obtained for $r\ge |V|/2$.  Tighter bounds can be obtained
by replacing $\alpha(G)$ by $\alpha'(G) := \max \{ c(G[J]) \mathrel|
\emptyset\neq J\subseteq V, \, B_J\neq\emptyset \}$.  In particular, if
$G[J]$ is connected whenever $B_J\neq\emptyset$ for any choice of $J\subseteq
V$, $J\neq\emptyset$, then $\alpha'(G)=1$.  In this case, for $r\ge
\frac{|V|+1}{2}$ the inequalities in (\ref{alt1}) and (\ref{mainineq}) turn
into an identity, which is well-known in abstract tube theory
\cite{Dohmen:2003}.
\end{remark}

\begin{remark}
\label{counterex}
The requirement that $G$ is chordal cannot be omitted from Theorem
\ref{maintheorem}.  Consider the non-chordal graph $G=(V,E)$ depicted in
Figure \ref{mycounterexample}, and consider events $A_v$, $v\in V$, with
$\Pr(A_v)=1$ for any $v\in V$.  The clique complex of this graph consists of 8
cliques of size 1, 20 cliques of size 2, and 16 cliques of size 3.  Since
$\alpha(G)=3$, the first inequality in Theorem \ref{maintheorem} gives the
non-valid bound $1 \ge \frac{1}{3}(8-20+16)=\frac{4}{3}$ for $r\ge 2$.
\par
This graph $G$ is a subgraph of $G_3$\,---\,the first graph in an infinite
sequence $(G_k)_{k=3,5,\dots}$ of non-chordal graphs for which
(\ref{mainineq}) does not hold.  Each $G_k$ is the join of $k$ disjoint copies
of the edgeless graph on three vertices.  By letting $r\ge \frac{3k}{2}$ and
$\Pr(A_v)=1$ for any vertex $v$ of $G_k$, (\ref{mainineq}) specializes to $1
\ge \frac{1}{3} (1 + 2^k)$, which does not hold for any $k\ge 3$.
\end{remark}

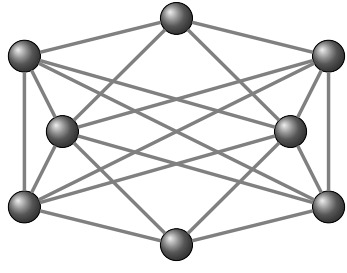
\begin{figure}[ht]
\begin{center}
\scalebox{0.5}{%
\begin{tikzpicture}[node distance = 3cm]
\GraphInit[vstyle=Shade]
\SetGraphShadeColor{gray}{gray}{gray}
\SetVertexNoLabel
\Vertex[x=4,y=0]{H}
\Vertex[x=0,y=1]{A}
\Vertex[x=0,y=5]{B}
\Vertex[x=1,y=3]{C}
\Vertex[x=4,y=6]{G}
\Vertex[x=8,y=1]{F}
\Vertex[x=8,y=5]{E}
\Vertex[x=7,y=3]{D}
\Edges(A,B,G,E,F,H,A)
\Edges(A,C,B)
\Edges(F,D,E)
\Edge(A)(D)
\Edge(A)(E)
\Edge(C)(E)
\Edge(C)(F)
\Edge(B)(D)
\Edge(B)(F)
\Edge(C)(H)
\Edge(D)(H)
\Edge(C)(G)
\Edge(D)(G)
\end{tikzpicture}
}
\end{center}
\caption{\label{mycounterexample} A non-chordal graph for which
  (\ref{mainineq}) does not hold.}
\end{figure}

\section{Particular cases}

Similar to \cite{Dohmen:2003}, where upper bounds are derived from
(\ref{alt1}) for different choices of $G$, lower bound analogues of these
bounds can be derived from (\ref{mainineq}).  As examples we consider lower
bound analogues of Hunter's inequality \cite{Hunter:1976} and Seneta's
inequality \cite{Seneta:1988}.

\begin{corollary} 
\label{cortree1}
Let $\{A_v\}_{v\in V}$ be a non-empty finite collection of events,
where the indices form the vertices of a tree $G=(V,E)$. Then,
\begin{align*}
\Pr\left(\, \bigcup_{v\in V} A_v \right) & \,\ge\, \frac{1}{\alpha(G)} \left(\, \sum_{v\in V} \Pr(A_v) \, - \sum_{\{v,w\}\in E} \Pr(A_v\cap A_w) \right) .
\end{align*}
\end{corollary}

\begin{proof}
Since any tree is chordal, the result follows from Theorem \ref{maintheorem}
with $r\ge 2$.
\end{proof}

\begin{corollary}
\label{analoguetoseneta}
Let $A_1,\dots,A_n$ be a finite collection of events, and
$j,k\in\{1,\dots,n\}$. Then,
\begin{multline*}
\Pr\left(\,\bigcup_{i=1}^n A_i \right) \,\ge\, \frac{1}{n-2+\delta_{jk}} \left( \sum_{i=1}^n \Pr(A_i) - \sum_{\substack{i=1\\ i\neq j}}^n \Pr(A_i\cap A_j) - \sum_{\substack{i=1\\ i\neq j,k}}^n \Pr(A_i\cap A_k) \right. \\
\left. + \sum_{\substack{i=1\\ i\neq j,k}}^n \Pr(A_i\cap A_j\cap A_k) \right)
,
\end{multline*}
provided $n>2-\delta_{jk}$ where $\delta_{jk}$ denotes the Kronecker delta.
\end{corollary}

\begin{proof}
Define $G = K_{2-\delta_{jk}} \ast L_{n-2+\delta_{jk}}$ (the join of a
complete graph on $2-\delta_{jk}$ vertices and an edgeless graph on
$n-2+\delta_{jk}$ vertices), and apply Theorem \ref{maintheorem} with $r\ge |V|/2$.
\end{proof}

\section*{Acknowledgment}

The author wishes to thank Jakob Jonsson (KTH Stockholm) for pointing out the
sequence of graphs in Remark~\ref{counterex}.


\begin{thebibliography}{99}

\bibitem{Boros-Veneziani:2002} E. Boros \&\ P. Veneziani, \emph{Bounds of
  degree 3 for the probability of the union of events}, Rutcor Research Report
3-02, 2002.

\bibitem{Diestel:2005} R. Diestel, \emph{Graph Theory}, Graduate Texts in
Mathematics, 3rd edition, Springer-Verlag, New York, 2005.

\bibitem{Dohmen:2003} K. Dohmen, \emph{Improved Bonferroni Inequalities via
  Abstract Tubes}, Lecture Notes in Mathematics No\@. 1826, Springer-Verlag,
Berlin/Heidelberg, 2003.

\bibitem{Dohmen:2002} K. Dohmen, \emph{Bonferroni-type inequalities via
  chordal graphs}, Combin\@. Probab\@. Comput\@. \textbf{11} (2002), 349--351.

\bibitem{Edelman-Reiner:2000} P.H. Edelman and V. Reiner, \emph{Counting the
  interior points of a point configuration}, Discrete
Comput\@. Geom\@. \textbf{23} (2000), 1--13.

\bibitem{Galambos-Simonelli:1996} J. Galambos and I. Simonelli,
\emph{Bonferroni-type Inequalities with Applications}, Springer-Verlag, New
York, 1996.

\bibitem{Hunter:1976} D. Hunter, \emph{An upper bound for the probability of a
  union}, J\@. Appl\@. Prob\@. \textbf{13} (1976), 597--603.

\bibitem{Naiman-Wynn:1997} D.Q. Naiman and H.P. Wynn, \emph{Abstract tubes,
  improved inclusion-exclusion identities and inequalities and importance
  sampling}, Ann\@. Statist\@. \textbf{25} (1997), 1954--1983.

\bibitem{Seneta:1988} E. Seneta, \emph{Degree, iteration and permutation in
  improving Bonferroni-type bounds}, Austral\@. J\@. Statist\@. \textbf{30A}
(1988), 27--38.

\bibitem{Veneziani:2009} P. Veneziani, \emph{Upper bounds of degree 3 for the
  probability of the union of events via linear programming}, Discrete
Appl\@. Math\@. \textbf{157} (2009), 858--863.

\bibitem{Saenz-Wynn:2009} E. S\'aenz-de-Cabez\'on and H.P. Wynn,
\emph{Betti numbers and minimal free resolutions for multi-state system
  reliability bounds}, J. Symb\@. Comp\@. \textbf{44} (2009), 1311--1325.

\end{thebibliography}
\end{document}